\documentclass[reqno,12pt]{amsart}
\usepackage{geometry}
\usepackage{epsfig}
\usepackage{amsmath}
\usepackage{amsfonts}
\usepackage{mathrsfs}
\usepackage{subfigure}
\usepackage{cite}
\usepackage[colorlinks,
            linkcolor=blue,
            anchorcolor=blue,
            citecolor=blue
            ]{hyperref}

\newtheorem{theorem}{Theorem}[section]

\newtheorem{corollary}[theorem]{Corollary}

\newtheorem{pro}[theorem]{Proposition}
\newtheorem{remark}[theorem]{Remark}
\newtheorem{open problem}[theorem]{Open Problem}

\newcommand{\IN}{\mathbb{N}}

\newcommand{\e}{\mbox{e}}
\newcommand{\K}{\mathscr{K}}
\newcommand{\E}{\mathscr{E}}

\newcommand{\s}{\sum_{n=0}^{\infty}}
\newcommand{\su}{\sum_{n=1}^{\infty}}
\newcommand{\la}{\lambda}
\newcommand{\G}{\Gamma}

\newcommand{\p}{\psi}
\newcommand{\va}{\varphi}

\numberwithin{equation}{section}
\allowdisplaybreaks[4]

\begin{document}

\def\square{\hfill${\vcenter{\vbox{\hrule height.4pt \hbox{\vrule width.4pt
height7pt \kern7pt \vrule width.4pt} \hrule height.4pt}}}$}

\title[]
{\small Notes on Generalized Gr\"otzsch Ring Function and Generalized Hersch-Pfluger Distortion Function$^{\ast}$}


\author{Qi Bao$^{1}$}

\address{Qi Bao$^{1}$, $^{1}$School of Mathematics and Statistics, Huangshan University, Huangshan 245021, Anhui, China}

\email{600218@hsu.edu.cn}


\author{Miao-Kun Wang$^{2,\ast\ast}$}

\address{Miao-Kun Wang$^{2}$, $^{2}$Department of Mathematics, Huzhou University, Huzhou 313000, Zhejiang, China}

\email{wmk000@126.com, wangmiaokun@zjhu.edu.cn}

\subjclass[2010]{11F03, 33C05}

\keywords{generalized Gr\"otzsch ring function; generalized Ramanujan's modular equation; generalized Hersch-Pfluger distortion function; quasiconformal mappings}

\thanks{$^{\ast}$ This research was supported by Zhejiang Provincial Natural Science Foundation of China under Grant No. LY24A010011.}

\thanks{$^{\ast\ast}$Corresponding author. Email address: wmk000@126.com}

\begin{abstract}
For $a\in(0,1)$, $r\in(0,1)$ and $K\in(1,\infty)$, let $\mu_{a}(r)$ and $\varphi_{K}^{a}(r)$ be the generalized Gr\"{o}tzsch ring function and generalized Hersch-Pfluger distortion function. In the past few years, the functions $\mu_{a}(r)$ and $\varphi_{K}^{a}(r)$, and their special cases $\mu_{1/2}(r)$ and $\varphi_{K}^{1/2}(r)$ have been playing the very important role on the theory of quasiconformal mappings and (generalized) Ramanujan's modular equations. In this paper, we present a series expansion of $\mu_{a}(r)$, and thus prove that the function $r\mapsto -[\mu_{a}(r)-\log{(e^{R(a)/2})/r}]$  is absolutely monotonic on $(0,1)$. Here $R(a)$ is the Ramanujan constant. In addition, we also investigate the submultiplicative and power  submultiplicative properties of $\varphi_{K}^{a}(r)$, and establish some new inequalities for $\varphi_{K}^{a}(r)$ in terms of elementary functions.
\end{abstract}

\maketitle
\section{Introduction}
In geometric function theory, the solutions of many extremal problems are conformal invariants, and they have expressions in terms of special functions. For example,  the Gr\"otzsch ring function, which is a ratio of the complete elliptic integral of the first kind, represents the conformal modulus of the Gr\"otzsch ring $B^2\backslash [0,r]$ for $r\in(0,1)$, where $B^2$ is the unit disk in the complex plane. The Hersch-Pfluger distortion function is defined by use of the Gr\"otzsch ring function and its inverse function, which gives the maximum distortion of the class of $K$-quasiconformal mappings of the unit disk onto itself with the origin fixed. It is well known that the above functions are the most fundamental and important functions in the theory of quasiconformal mappings in the plane, for details, please refer to \cite{Ah, AVVb, Lehto2}.

Throughout this paper, we always let $r'=\sqrt{1-r^2}$ for $r\in [0,1]$, $\IN$ denote by the set of positive integers and $\infty$ the positive infinity. For complex number $x$ with ${\rm Re}\,x>0$, let
\begin{align*}
\Gamma(x)=\int_0^{\infty}t^{x-1}\e^{-t}dt, \,\,
\psi(x)=\frac{\Gamma'(x)}{\Gamma(x)}
\end{align*}
be the classical Euler gamma and psi (digamma) functions, respectively (cf. \cite[6.1.1-6.2.1]{AS}). For complex numbers $a, b$ and $c$ with $c\neq 0,-1,-2,\cdots$, the Gaussian hypergeometric function is defined by
(cf. \cite[15.1.1]{AS})
\begin{align}\label{2F1}
F(a,b;c;x)={}_2F_1(a,b;c;x)=\sum_{n=0}^{\infty}
\frac{(a)_n(b)_n}{(c)_n} \frac{x^n}{n!},\,\,|x|<1,
\end{align}
where $(a)_n$ is the Pochhammer symbol or shifted factorial defined as $(a)_0=1$, and $(a)_n=a(a+1)(a+2)\cdots (a+n-1)=\G(n+a)/\G(a)$ for $n\in\IN$. $F(a,b;c;x)$ is zero-balanced if $c=a+b$, and the  zero-balanced hypergeometric function satisfies the following Ramanujan's asymptotic formula:
\begin{equation*}
B(a,b)F(a,b;a+b;x)+\log(1-x)=R(a,b)+O((1-x)\log(1-x)), \quad x\rightarrow 1.
\end{equation*}
Here $B(a,b)\equiv  \Gamma(a)\Gamma(b)/\Gamma(a+b)$ is the beta function and 
(cf. \cite{qiu2017})
\begin{align}\label{R(a,b)}
R(a,b)\equiv -2\gamma-\p(a)-\p(b), 
\end{align}
where $\gamma=\lim \limits_{n\to \infty} (\sum_{k=1}^n 1/k-\log n)=0.577215\cdots$ is the Euler-Mascheroni constant. Let $a\in(0,1/2]$ and $b=1-a$, formula \eqref{R(a,b)} degenerates into the so-called Ramanujan constant or Ramanujan $R$-function (cf. \cite{QMH2020})
\begin{align*}
R(a) \equiv R(a,1-a)= -2\gamma-\p(a)-\p(1-a).
\end{align*}

As a ratio of two Gaussian hypergeometric functions, the generalized Gr\"otzsch ring function
\begin{align}\label{uar}
\mu_a(r)=\frac{\pi}{2\sin(\pi a)}\frac{ F\left(a,1-a;1; 1-r^2\right) }{F \left(a,1-a;1;r^2\right)},\quad a\in(0,1)
\end{align}
is introduced by Qiu and Vuorinen in 1999 (cf. \cite{QV1999}), which extend the classical Gr\"otzsch ring function $\mu(r) \equiv \mu_{1/2}(r)$. As early as 1829, Jacobi proved
\begin{align}\label{Jacobi}
{\rm e}^{\mu(r)+\log r}=4 \prod_{n=1}^{\infty}
\left(\frac{1+q^{2n}}{1+q^{2n-1}}\right)^4
\end{align}
for $r\in(0,1)$, where $q={\rm e}^{-2\mu(r)}$ (cf. \cite{Jacobi}).   Qiu and Vuorinen \cite{QV1999} established an infinite product inequality for generalized Gr\"otzsch ring function and then derived another infinite product identity for the Gr\"otzsch ring function. Specifically speaking, for $a\in(0,1/2]$ and $r\in(0,1)$, let
\begin{align*}
r_0=r'=\sqrt{1-r^2} \quad \text{and}\quad  r_n=\frac{2\sqrt{r_{n-1}}}{1+r_{n-1}}.
\end{align*}
Then they showed
\begin{align}\label{ua+logr}
\prod_{n=0}^{\infty} (1+r_n)^{2^{-n}} \leq {\rm e}^{ \mu_a(r)+\log r }
\leq \frac{ {\rm e}^{\frac{R(a)}{2}} }{4} \prod_{n=0}^{\infty} (1+r_n)^{2^{-n}},
\end{align}
with equality in each instance if and only if $a=1/2$, that is,
\begin{align*}
{\rm e}^{ \mu(r)+\log r }=\prod_{n=0}^{\infty} (1+r_n)^{2^{-n}}.
\end{align*}
The above equation is more convenient to calculate the value of $\mu(r)$ than (\ref{Jacobi}), since the right-hand side of the latter is involved in $\mu(r)$. 
Very recently, Yang, Wang and Tian \cite{Yang-Wang-Tian-2024} found another infinite series expansion formula for $\mu(r)$, and thus obtained some higher-order monotonicity properties for certain functions involving the Gr\"{o}tzsch ring function. 

In this paper, by studying the derivative of the generalized  Gr\"otzsch ring function $\mu_{a}(r)$ with respect to $r$, a new infinite series expansion formula for $\mu_{a}(r)$ will be given (see Theorem \ref{iduar}). As by-product, a conjecture on the absolute monotonicity of a function involving the Gaussian hypergeometric function will be proved to be true.

Let $a\in(0,1)$, $r\in(0,1)$ and $p\in(0,\infty)$. Then the Ramanujan's generalized modular equation with signature $1/a$ and order (or degree) $p$ is
\begin{align}\label{Ramanujan}
\frac{F(a,1-a;1;1-s^2)}{F(a,1-a;1;s^2)}
=p\, \frac{F(a,1-a;1;1-r^2)}{F(a,1-a;1;r^2)}
\end{align}
(cf. \cite{Be3,Be5,Be6}). In particular, when $a=1/2$, equation \eqref{Ramanujan} reduces to the classical modular equation, thus the word ``generalized" alludes to the fact that the parameter $a\in(0,1)$ is arbitrary. The Ramanujan's generalized modular equation (\ref{Ramanujan}) has been developed by leading mathematicians for over a century. The Ramanujan classical modular equation  was firstly studied by Jacobi in the nineteenth century. In 1920s, Ramanujan extensively studied modular equations (cf. \cite{Be3,Be6}) and provided numerous algebraic identities for the solutions $s$ of (\ref{Ramanujan}) at rational values of $a$ such as $1/6$, $1/4$, $1/3$. Borwein and Borwein, Venkatachaliengar, and Berndt et al. made great contribution to this subject, see \cite{Borwein1,Borwein2,CL1998,VD,Be3,Be5,Be6,Liu,CL2003,BB2007,R1990,V2015,Liu2003} for more details.

It is worth noting that Anderson, Qiu, Vamanamurthy and Vuorinen \cite{Vuorinen1996,AQVV} were surprised to find the relationship between the generalized Gr\"otzsch ring function $\mu_a(r)$ and the generalized Ramanujan's modular equation. Indeed, by (\ref{uar}), we can rewrite (\ref{Ramanujan}) as
\begin{align}\label{Ramanujan1}
\mu_a(s)=p \mu_a(r),
\end{align}
so that the solution of Ramanujan generalized modular equation (\ref{Ramanujan1}) is then given by
\begin{align}
s=\varphi_K^a(r)\equiv \mu_a^{-1}\left(\frac{\mu_a(r)}{K}\right),
\,\, p=\frac1K,
\end{align}
which is called generalized Hersch-Pfluger distortion function (cf. \cite{AQVV,WQCJ2012}). Obviously, as a function of $r$, $\varphi_K^a(r)$ is an increasing homeomorphism from $[0,1]$ onto itself. When $a=1/2$, $\varphi_K(r)\equiv \varphi_K^{1/2}(r)$ is exactly the well-known Hersch-Pfluger distortion function. In the past few decades, many elegant properties of $\varphi_K(r)$ and $\varphi_K^a(r)$ have been obtained, such as
\begin{align}\label{va+va'}
\varphi_K^a(r)^2+\varphi_{1/K}^a(r')^2=1
\end{align}
for all $a\in(0,1/2]$, $r\in(0,1)$ and $K\in(0,\infty)$, for more results see \cite{AVV1988,AVVb,AQVV,QVV1997,QV1996,WQCJ2012,QVV1999,P1991,P1993,WZC2007,WZCQ2008,WCS2015,WCZ2019,WCC2021,WLC2018}.

One of the meaningful tasks is to investigate the submultiplicative and power submultiplicative properties of Hersch-Pfluger distortion function. Following this topic, Anderson, Vamanamurthy and Vuorinen \cite{AVVb,AVV1988} proved that
\begin{align}\label{In1}
\frac{\va_K(r)\va_K(t)}{\va_K(rt) }\leq 4^{1-\frac1K} \quad \text{and} \quad
\frac{\va_K(r)^p}{\va_K(r^p)}<4^{(p-1)\left(1-\frac1K\right)}
\end{align}
hold for all $K\in(1,\infty)$, $p\in(1,\infty)$ and $r,t\in(0,1)$. Recently, Wang, Qiu and Chu \cite{WQC2018} substantially improved the above inequalities (\ref{In1}) and established some sharp inequalities concerning the submultiplicative and power submultiplicative properties for Hersch-Pfluger distortion function. They proved the following Theorems \ref{th1} and \ref{th2}. To do this, we introduce generalized H\"ubner function (cf. \cite{AQVV,WQCJ2012,QV1999})
\begin{align*}
m_a(r)=\frac{2}{\pi \sin(\pi a)}r'^2 F(a,1-a;1;r^2) F(a,1-a;1;1-r^2)
\end{align*}
for $a\in(0,1/2]$ and $r\in(0,1)$.
When $a=1/2$, the generalized H\"ubner function $m_a(r)$ degenerates into the famous H\"ubner function $m(r)\equiv m_{1/2}(r)$ (cf. \cite{MQJ2021,QVV1999,Hubner}).

\begin{theorem}\label{th1}
Let $a(r,t)$, $b(r,t)$ and $c(r,t)$ be real functions defined on $(0,1)\times(0,1)$. Then the following statements are true:

$(1)$ The inequality
\begin{align}
\frac{\va_K(r)\va_K(t)}{\va_K(rt)}<{\rm e}^{a(r,t)\left(1-\frac1K\right)}
\end{align}
holds for all $r,t\in(0,1)$ and $K\in[1,\infty)$ if and only if $a(r,t)\geq m(r)+m(t)-m(rt)$.

$(2)$ The inequality
\begin{align}
{\rm e}^{c(r,t)(1-K)}<\frac{\va_{1/K}(r)\va_{1/K}(t)}{\va_{1/K}(rt)}
<{\rm e}^{b(r,t)(1-K)}
\end{align}
holds for all $r,t\in(0,1)$ and $K\in(1,\infty)$ if and only if $b(r,t)\leq m(r)+m(t)-m(rt)$ and $c(r,t)\geq \mu(r)+\mu(t)-\mu(rt)$.
\end{theorem}

\begin{theorem}\label{th2}
Let $A(r)$, $B(r)$ and $C(r)$ be real functions defined on $(0,1)$. Then the following statements are true:

$(1)$ If $0<p\leq1$, then the inequalities
\begin{align*}
\frac{\va_K(r)^p}{\va_K(r^p)}>{\rm e}^{A(r)\left(1-\frac1K\right)}
\quad \text{and} \quad
{\rm e}^{B(r)(1-K)}<\frac{\va_{1/K}(r)^p}{\va_{1/K}(r^p)}<{\rm e}^{C(r)(1-K)}
\end{align*}
hold for all $r\in(0,1)$ and $K\in(1,\infty)$ if and only if $A(r)\leq pm(r)-m(r^p)$, $B(r)\geq pm(r)-m(r^p)$, $C(r)\leq p\mu(r)-\mu(r^p)$, respectively.

$(2)$ If $p>1$, then the inequalities
\begin{align*}
\frac{\va_K(r)^p}{\va_K(r^p)}<{\rm e}^{A(r)\left(1-\frac1K\right)}
\quad \text{and} \quad
{\rm e}^{B(r)(1-K)}<\frac{\va_{1/K}(r)^p}{\va_{1/K}(r^p)}<{\rm e}^{C(r)(1-K)}
\end{align*}
hold for all $r\in(0,1)$ and $K\in(1,\infty)$ if and only if $A(r)\geq pm(r)-m(r^p)$, $B(r)\geq p\mu(r)-\mu(r^p)$, $C(r)\leq pm(r)-m(r^p)$, respectively.
\end{theorem}
Based on the above mentioned, the second purpose of this paper is to extend Theorems \ref{th1} and \ref{th2} to the case of generalized Hersch-Pfluger distortion function $\varphi_K^a(r)$ (see Theorems \ref{th3}-\ref{th4}). With the obtained results, new inequalities concerning  submultiplicative and power submultiplicative properties of $\varphi_K^a(r)$ will be established.

The rest of paper is organized as follows. Section \ref{Sec2} mainly presents the infinite series expansion of generalized Gr\"otzsch ring function, and gives an affirmative answer to a conjecture proposed by Tian and Yang in \cite{TY2023}. In Section 3, we first prove the submultiplicative and power submultiplicative properties for generalized Hersch-Pfluger distortion function, and then derive several inequalities for it in terms of elementary functions.

\medskip
\section{Series expansion and asymptotic formula of $\mu_{a}(r)$}\label{Sec2}

Before establishing the series expansion of $\mu_{a}(r)$, we first prove the following Proposition \ref{lemma2}, which not only has independent meaning but also serves as a preparation lemma for Theorem \ref{iduar}.


\begin{pro}\label{lemma2}
For arbitrarily given $a\in(0,1/2]$, and for $x\in(0,1)$, the coefficients of the Maclaurin series of the function $x\mapsto1/[\sqrt{1-x} F(a,1-a;1;x)]$ are all nonnegative.
\end{pro}

\begin{proof}
For each $a\in(0,1/2]$, since $x\mapsto 1/[\sqrt{1-x}F(a,1-a;1;x)]$ is analytic in a neighborhood of $x=0$, then there exists a sequence $\{\lambda_n\}$ independent of $x$ such that
\begin{equation}\label{eq2.1}
\frac{1}{\sqrt{1-x} \, F(a,1-a;1;x)}=\s \lambda_n x^n.
\end{equation}
Using (\ref{2F1}), the above equation can be further rewritten as
\begin{align}\label{2eq1}
\frac{1}{\sqrt{1-x}}=\s b_n x^n \s \lambda_n x^n=\s \sum_{k=0}^n \lambda_k b_{n-k} x^n
\end{align}
with
\begin{align}
b_n=\frac{(a)_n(1-a)_n}{(n!)^2}>0.
\end{align}

Let
\begin{align}\label{Wn}
	W_n=\frac{(1/2)_n}{n!}
\end{align}
be the Wallis ratio, then by \cite[2.1.6]{AS} one has
\begin{align}
\frac{1}{\sqrt{1-x}}=\s W_n x^n.
\end{align}
This, in conjunction with \eqref{2eq1}, gives that
\begin{align*}
\s W_n x^n= \s \sum_{k=0}^n \lambda_k b_{n-k} x^n.
\end{align*}
Equating the coefficients of $x^n$ in the above equation, we obtain that $\lambda_0=1$ and
\begin{align}\label{2eq2}
W_n = \sum_{k=0}^n \lambda_k b_{n-k}=\lambda_n b_0+ \sum_{k=0}^{n-1} \lambda_k b_{n-k}, \,\, n\geq1.
\end{align}
Noting that $b_0=1$, and using (\ref{2eq2}), one can derive
\begin{align}\label{an}
\lambda_n=W_n-\sum_{k=0}^{n-1} \lambda_k b_{n-k}, \,\, n\geq1.
\end{align}

Next, we prove that $\lambda_{n}>0$ for $n\in \mathbb{N}$ by mathematical induction. Firstly, from \eqref{Wn} and \eqref{an} it is not difficult to see that
\begin{align*}
 \lambda_1=a^2-a+\frac12>0 \quad \text{and} \quad
\lambda_2=\frac34a^4-\frac32a^3+\frac74a^2-a+\frac38>0
\end{align*}
for all $a\in(0,1/2]$. Secondly, since $W_{n}=(2n-1)/(2nW_{n-1})$ and by \eqref{an},
\begin{align}\label{an-1}
\lambda_{n-1}=W_{n-1}-\sum_{k=0}^{n-2} \lambda_k b_{n-k-1}.
\end{align}
Eliminating $W_{n-1}$ from (\ref{an}) and (\ref{an-1}) results in
\begin{align}\label{an(1)}
\lambda_n=\frac{2n-1}{2n} \lambda_{n-1}-\sum_{k=0}^{n-1} \lambda_k b_{n-k}+\frac{2n-1}{2n} \sum_{k=0}^{n-2} \lambda_k b_{n-k-1}.
\end{align}
By
\begin{align}
b_{n-k}=\frac{(n+a-k-1)(n-k-a)}{(n-k)^2}b_{n-k-1},
\end{align}
(\ref{an(1)}) can be rewritten as
\begin{align}\label{an(2)}
\lambda_n=\left( \frac{2n-1}{2n}-a(1-a) \right) \lambda_{n-1}+\sum_{k=0}^{n-2}
\frac{n^2+2a(a-1)n-k^2}{2n(n-k)^2} \lambda_k b_{n-k-1}.
\end{align}
Finally, it suffices to prove that $(2n-1)/(2n)-a(1-a)\geq 0$ and $n^2+2a(a-1)n-k^2\geq 0$ for $n\geq 2$, $a\in(0,1/2]$ and $0\leq k \leq n-2$. Indeed, for $n\geq 2$ and $a\in(0,1/2]$, it is easy to calculate that
\begin{align*}
 \frac{2n-1}{2n}-a(1-a)\geq \frac56-a(1-a)\geq \frac{7}{12}
\end{align*}
and
\begin{align*}
n^2+2a(a-1)n-k^2 \geq [4-2a(1-a)]n-4 >0, \quad 0\leq k\leq n-2.
\end{align*}
This completes the proof.
\end{proof}

\begin{theorem}\label{iduar}
For each $a\in(0,1/2]$, let $\{\lambda_n\}$ be the positive sequence defined by (\ref{an}), and
\begin{align*}
\theta_n=\frac{1}{2n}\sum_{k=0}^n \lambda_k \lambda_{n-k}, \quad n\geq 1.
\end{align*}
Then, for $r\in(0,1)$, we have
\begin{align*}
\mu_a(r)+\log r=\frac{R(a)}{2}-\su \theta_n r^{2n}.
\end{align*}
\end{theorem}

\begin{proof}
By differentiation (see \cite[Theorem 4.1(5)]{AQVV}),
\begin{align*}
\frac{d [\mu_a(r)+\log r]}{dr}
=\frac 1r \left[ 1- \left(\frac{1}{\sqrt{1-r^2}F\left(a,1-a;1; r^2\right)  }\right)^2 \right].
\end{align*}
Making use of \eqref{eq2.1} and $\lambda_{0}=1$ gives
\begin{align*}
\frac{d [\mu_a(r)+\log r]}{dr}
&=\frac{1}{r} \left[ 1-\left( \sum_{n=0}^{\infty} \lambda_n r^{2n}\right)^2\right]=\frac{1}{r} \left( 1-\sum_{n=0}^{\infty} \sum_{k=0}^n \lambda_k \lambda_{n-k} r^{2n} \right) \nonumber \\
&=- \sum_{n=1}^{\infty} \sum_{k=0}^n \lambda_k \lambda_{n-k} r^{2n-1}.
\end{align*}
Next integrating from $0$ to $r$, we find that
\begin{align*}
\int_0^r  \frac{d [\mu_a(x)+\log x]}{dx} dx=-\int_0^r \su \sum_{k=0}^n \lambda_k \lambda_{n-k} x^{2n-1} dx.
\end{align*}

Proposition \ref{lemma2} shows that $\left\{\sum_{k=0}^n \lambda_k \lambda_{n-k} x^{2n-1}\right\}_{n=1}^{\infty}$ is non-negative measurable function sequence on $x\in(0,1)$. Using Lebesgue theorem, we obtain
\begin{align*}
\int_0^r  \frac{d [\mu_a(x)+\log x]}{dx} dx = - \su \sum_{k=0}^n \lambda_k \lambda_{n-k} \int_0^r x^{2n-1} dx.
\end{align*}
So that
\begin{align*}
\mu_a(r)+\log r-\frac{R(a)}{2}=
- \su  \sum_{k=0}^n \frac{ \lambda_k \lambda_{n-k} } {2n} r^{2n}.
\end{align*}
Here we have used the limiting value $\lim_{r\rightarrow 0^+}\mu_{a}(r)+\log{r}=R(a)/2$ (see \cite[Theorem 5.5\,(2)]{AQVV}). This completes the proof.
\end{proof}

\begin{remark}
A function $f$ is called absolutely monotonic on an interval $I$
if it has nonnegative derivatives of all orders on $I$, that is,
$f^{\left( n\right) }\left( x\right) \geq 0$
for $x\in I$ and every $n\geq 0$ (see \cite{Widder-LT-1946}). In particular,  if $
f(x)=\sum_{n=0}^{\infty}a_{n}x^n$ is a power series converging on $(0,r)$ $(r>0)$, then $f(x)$ is
absolutely monotonic on $(0,r)$ if and only if $a_{n}\geq 0$ for all $n\geq 0$. Proposition \ref{lemma2} and Theorem \ref{iduar} imply that $r \mapsto {R(a)}/{2}-\mu_a(r)-\log r$ is absolutely monotonic on $(0,1)$. As a conclusion, $r\mapsto \mu_a(r)+\log r$ is strictly decreasing and concave on $(0,1)$.
\end{remark}

Employing Proposition \ref{lemma2}, a conjecture due to Tian and Yang in \cite{TY2023}, which states that $x \mapsto -[1/F\left(a,1-a;1; x\right)]'$ is absolutely monotonic on $(0,\infty)$, can be confirmed now.


\begin{theorem}\label{contheorem}
For each $a\in(0,1)$, the function $x \mapsto -[1/F\left(a,1-a;1; x\right)]'$ is absolutely monotonic on $(0,\infty)$.
\end{theorem}
\begin{proof}
Let
\begin{equation*}
f(x)=-\left[\frac{1}{F(a,1-a;1;x)}\right]',
\end{equation*}
then by simple computation one has
\begin{align}\label{eqf(x)}
f(x)=&\frac{a(1-a)F(a+1,2-a;2;x)}{[F(a,1-a;1;x)]^2}
=\frac{a(1-a)F(a,1-a;2;x)}{(1-x)[F(a,1-a;1;x)]^2} \nonumber\\
=&a(1-a)\left[\frac{1}{\sqrt{1-x}F(a,1-a;1;x)}\right]^2F(a,1-a;2;x).
\end{align}
	
Therefore, Theorem \ref{contheorem} directly follows from \eqref{eqf(x)} and Proposition \ref{lemma2}.
\end{proof}

In addition, it is worth pointing out that, since the values of $\{\theta_{n}\}$ in Theorem \ref{iduar} are computable,  then Theorem \ref{iduar} provides a algorithm to obtain high-precision estimates of $\mu_{a}(r)$. For example, use of (\ref{an}) gives
\begin{align*}
&\lambda_0=1, \quad \lambda_1=a^2-a+\frac12, \quad \lambda_2=\frac34a^4-\frac32a^3+\frac74a^2-a+\frac38, \\
&\lambda_3=\frac{19}{36}a^6-\frac{19}{12}a^5+\frac{197}{72}a^4-\frac{17}6a^3+\frac{19}9a^2-\frac{23}{24}a+\frac5{16}.
\end{align*}
Then by Theorem \ref{iduar}, we obtain
\begin{align*}
\theta_1=\lambda_0 \lambda_1=a^2-a+\frac12,
\end{align*}
\begin{align*}
\theta_2=\frac{\lambda_1^2+2 \lambda_0\lambda_2}{4}=\frac54a^4-\frac52a^3+\frac{11}4a^2-\frac32a+\frac12,
\end{align*}
\begin{align*}
\theta_3=\frac{\lambda_0 \lambda_3 + \lambda_1\lambda_2}{3}=\frac{23}{54}a^6-\frac{23}{18}a^5+\frac{229}{108}a^4-\frac{19}9a^3
+\frac{157}{108}a^2-\frac{11}{18}a+\frac16.
\end{align*}
Furthermore, we have
\begin{corollary}\label{thua(r)}
For each $a\in (0,1/2]$, let
\begin{align*}
\delta=\frac{R(a)}{2}-\sum_{k=1}^{3}\theta_{k}=\frac{R(a)}{2}- \left(\frac{23}{54}a^6-\frac{23}{18}a^5+\frac{91}{27}a^4-\frac{83}{18}a^3+\frac{281}{54}a^2-\frac{28}9a+\frac76\right).
\end{align*}
Then the function
\begin{align*}
H(r)=\frac{R(a)}{2}-\mu_a(r)-\log r-\theta_1 r^2-\theta_2 r^4-\theta_3 r^6
\end{align*}
is strictly increasing and convex from $(0,1)$ onto $(0,\delta)$. Consequently, for $a\in (0,1/2]$ and $r\in (0,1)$,
\begin{equation}\label{ua(r)ineq}
\sum_{k=1}^{3}\theta_{k}-\theta_1 r^2-\theta_2 r^4-\theta_3 r^6\leq \mu_a(r)+\log r \leq \frac{R(a)}{2}-\theta_1 r^2-\theta_2 r^4-\theta_3 r^6.
\end{equation}
The first (second) equality holds if and only if $r \rightarrow 1$ $(r \rightarrow 0, respectively)$.
\end{corollary}



\medskip

\section{Submultiplicative and power submultiplicative properties of generalized Hersch-Pfluger distortion function}
This section deals mostly with the inequalities involving the generalized Hersch-Pfluger distortion functions. To be more precise, we shall establish some sharp inequalities concerning the submultiplicative and power submultiplicative properties of $\varphi_{K}^{a}(r)$. Our main theorems are

\begin{theorem}\label{th3}
Let $\alpha(a,r,t)$, $\beta(a,r,t)$ and $\gamma(a,r,t)$ be  real functions defined on $(0,1/2]\times(0,1)\times(0,1)$. Then we have

$(1)$ The inequality
\begin{align}\label{In7}
\frac{\va_K^a(r)\va_K^a(t)}{\va_K^a(rt)}
<{\rm e}^{\alpha(a,r,t)\left(1-\frac1K\right)}
\end{align}
holds for all $(a,r,t)\in(0,1/2]\times(0,1)\times(0,1)$ and $K\in(1,\infty)$ if and only if $\alpha(a,r,t)\geq m_a(r)+m_a(t)-m_a(rt)$.

$(2)$ The inequality
\begin{align}
{\rm e}^{\gamma(a,r,t)(1-K)}
<\frac{\va_{1/K}^a(r)\va_{1/K}^a(t)}{\va_{1/K}^a(rt)}
<{\rm e}^{\beta(a,r,t)(1-K)}
\end{align}
holds for all $(a,r,t)\in(0,1/2]\times(0,1)\times(0,1)$ and $K\in(1,\infty)$ if and only if $\beta(a,r,t)\leq m_a(r)+m_a(t)-m_a(rt)$ and $\gamma(a,r,t)\geq \mu_a(r)+\mu_a(t)-\mu_a(rt)$.
\end{theorem}

\begin{theorem}\label{th4}
Let $\delta(a,r)$, $\zeta(a,r)$ and $\eta(a,r)$ be real functions defined on $(0,1/2]\times(0,1)$. Then we have

$(1)$ If $0<p<1$, then the inequalities
\begin{equation}\label{th4eq1}
\frac{\va_K^a(r)^p}{\va_K^a(r^p)}>{\rm e}^{\delta(a,r)\left(1-\frac1K\right)}
\quad \text{and} \quad
{\rm e}^{\zeta(a,r)(1-K)}
<\frac{\va_{1/K}^a(r)^p}{\va_{1/K}^a(r^p)}
<{\rm e}^{\eta(a,r)(1-K)}
\end{equation}
hold for all  $(a,r)\in(0,1/2]\times(0,1)$ and $K\in(1,\infty)$ if and only if $\delta(a,r)\leq pm_a(r)-m_a(r^p)$, $\zeta(a,r)\geq pm_a(r)-m_a(r^p)$, $\eta(a,r)\leq p\mu_a(r)-\mu_a(r^p)$, respectively.

$(2)$ If $p>1$, then the inequalities
\begin{equation}\label{th4eq2}
\frac{\va_K^a(r)^p}{\va_K^a(r^p)}<{\rm e}^{\delta(a,r)\left(1-\frac1K\right)}
\quad \text{and} \quad
{\rm e}^{\zeta(a,r)(1-K)}
<\frac{\va_{1/K}^a(r)^p}{\va_{1/K}^a(r^p)}
<{\rm e}^{\eta(a,r)(1-K)}
\end{equation}
hold for all  $(a,r)\in(0,1/2]\times(0,1)$ and $K\in(1,\infty)$ if and only if $\delta(a,r)\geq pm_a(r)-m_a(r^p)$, $\zeta(a,r)\geq p\mu_a(r)-\mu_a(r^p)$, $\eta(a,r)\leq pm_a(r)-m_a(r^p)$, respectively.
\end{theorem}

Before proving our main results in this section, let us introduce the Borwein's generalized elliptic integrals and prove several propositions at first: For $a\in(0,1)$ and $r\in(0,1)$, the generalized elliptic integrals of the first and second kind are defined as
\begin{equation}
\label{Ka}
\K_a=\K_a(r)=\frac{\pi}{2}\, F\left(a,1-a;1;r^2\right),
\K_a(0)=\frac{\pi}{2}, \K_a(1^-)=\infty,
\end{equation}
and
\begin{equation}\label{Ea}
\E_a=\E_a(r)=\frac{\pi}{2} \, F\left(a-1,1-a;1;r^2\right),\\
\E_a(0)=\frac{\pi}{2},
\E_a(1)=\frac{\sin(\pi a)}{2(1-a)},
\end{equation}
respectively (cf. \cite{AVVb,AQVV,V2009}).
The particular cases $\K\equiv\K_{1/2}$ and $\E\equiv\E_{1/2}$ are the well-known complete elliptic integrals of the first and second kinds, respectively (cf. \cite[17.3.9-17.3.10]{AS}). Following the above notation, the generalized Gr\"otzsch ring function and generalized H\"ubner function can be expressed by
\begin{align}
	\mu_a(r)&=\frac{\pi}{2\sin(\pi a)}\frac{ \K_a(r') }{ \K_a(r) },\label{eqmua}\\
	m_a(r)&=\frac{2}{\pi \sin(\pi a)}r'^2 \K_a(r') \K_a(r).\label{eqma}
\end{align}
By the symmetry of $a$ in \eqref{Ka}, \eqref{eqmua} and \eqref{eqma}, we can assume that $a\in(0,1/2]$ in the sequel. The following derivative formulas will be used frequently: For $a\in(0,1/2]$, $r\in(0,1)$ and $K\in(0,\infty)$, denote by $s=\va_K^a(r)$,
\begin{equation*}
\frac{d \K_a}{dr}=\frac{2(1-a)(\E_a-r'^2\K_a)}{rr'^2},\quad
\frac{d \E_a}{dr}=\frac{2(a-1)(\K_a-\E_a)}{r},
\end{equation*}
\begin{equation*}
\frac{d\mu_a(r)}{dr}=-\frac{\pi^2}{4rr'^2\K_a(r)^2},
\end{equation*}
\begin{equation*}
\frac{dm_a(r)}{dr}
=\frac{1}{r}+\frac{4\K_a(r)}{\pi r\sin(\pi a)}
\left[(1-2a)r^2\K_a(r')-2(1-a)\E_a(r')\right],
\end{equation*}
\begin{equation*}
\frac{\partial s}{\partial K}
=\frac{4ss'^2\K_a(s)^2}{\pi^2}\frac{\mu_a(r)}{K^2}
=\frac{4ss'^2\K_a(s)^2}{\pi^2}\frac{\mu_a(s)}{K}.
\end{equation*}

The proof of Theorems \ref{th3} and \ref{th4} requires the following preliminary results.

\begin{pro}\label{pro2} Let $a\in(0,1/2]$. Then

$(1)$ the function
\begin{align*}
f_1(r)\equiv \frac{2a-1}{1-a}r'^2\K_a(r)+2\E_a(r)
\end{align*}
is strictly decreasing from $(0,1)$ onto $\left(\sin(\pi a)/(1-a),{\pi}/[2(1-a)]\right)$.

$(2)$ the function
\begin{align*}
f_2(r)\equiv \frac{a(2a-1)}{1-a}r'^2\K_a(r)+\E_a(r)
\end{align*}
is strictly decreasing from $(0,1)$ onto $\left(\sin(\pi a)/[2(1-a)],\pi(2a^2-2a+1)/[2(1-a)]\right)$.

$(3)$ the function
\begin{align*}
f_3(r)\equiv \frac{r^2\K_a(r)}{\left(1+\frac{2a-1}
{2(1-a)}r^2\right)\K_a(r)-\E_a(r)}
\end{align*}
is strictly decreasing from $(0,1)$ onto $\left(2(1-a),{2(1-a)}/(2a^2-2a+1)\right)$.

\end{pro}

\begin{proof}
For part (1), using (\ref{Ka}) and (\ref{Ea}) to expand $f_{1}$ into power series, we obtain
\begin{align*}
f_1(r)&=\frac{\pi}{2(1-a)}
-\frac{\pi}{2}\su \left[\frac{(a)_{n-1}(1-a)_{n-1}}{(n-1)!^2}
 \frac{(2a^2-2a+1)n+a(a-1)}{(1-a)n^2}\right]r^{2n},
\end{align*}
from which the monotonicity and the limiting value of $f_1$ as $r\to0^+$ follow. By \cite[Lemma 5.4(1)]{AQVV} and (\ref{Ea}), $f_1(1^-)=\sin(\pi a)/(1-a)$.

For part (2), write  $f_2$ as $f_2(r)=af_1(r)+(1-2a)\E_a(r)$. Then by part (1) and (\ref{Ea}), one can easily obtain the monotonicity and the limiting values of $f_2$.


For part (3), applying (\ref{Ka}) and (\ref{Ea}), we obtain the following series expansion
\begin{align*}
\left(1+\frac{2a-1}{2(1-a)}r^2\right)\K_a(r)-\E_a(r)
=\frac{\pi}{2}\s \left[\frac{2a^2-2a+n+1}{2(1-a)(n+1)}
\frac{(a)_n(1-a)_n}{n!^2}\right]r^{2(n+1)},
\end{align*}
hence $f_3$ can be written as $f_3(r)=\s A_n r^{2n}/\s B_n r^{2n}$, where
\begin{align*}
A_n=\frac{(a)_n(1-a)_n}{n!^2} \quad \text{and} \quad
B_n=\frac{2a^2-2a+n+1}{2(1-a)(n+1)}\frac{(a)_n(1-a)_n}{n!^2}.
\end{align*}
Let $C_n=A_n/B_n$ for $n\geq 0$. Then
\begin{align*}
C_n-C_{n+1}=\frac{4a(1-a)^2}{(2a^2-2a+n+1)(2a^2-2a+n+2)}>0,
\end{align*}
which shows that the sequence $\{C_n\}$ is strictly decreasing. By \cite[Lemma 2.1]{Sp1}, $f_3$ is strictly decreasing on $(0,1)$. Obviously, $f_2(0^+)=C_0=A_0/B_0=2(1-a)/(2a^2-2a+1)$ and $f_2(1^-)=2(1-a)$.
\end{proof}

\begin{pro}\label{pro1}
For arbitrarily given $a\in(0,1/2]$ and for $p\in(0,\infty)$, define the functions $f_4$ and $f_5$ on $(0,1)$ by
\begin{align*}
f_4(r)= m_a(r^p)-pm_a(r) \quad \text{and}  \quad
f_5(r)= \mu_a(r^p)-p\mu_a(r),
\end{align*}
respectively. Then we have following conclusions:

$(1)$ If $p=1$, then $f_{4}(r)=f_{5}(r)\equiv 0$.

$(2)$ If $0<p<1$, then both $f_4$ and $f_5$ are strictly increasing from $(0,1)$ onto $(0,(1-p)R(a)/2)$.

$(3)$ If $p>1$, then both $f_4$ and $f_5$ are strictly decreasing from $(0,1)$
onto $((1-p)R(a)/2,0)$.
\end{pro}

\begin{proof}
The part (1) is obvious. For the function $f_4$, let $x=r^p$, then by differentiation,
\begin{align}\label{eq7}
\frac{d f_4(r)}{dr}=\frac{4p(1-a)}{\pi r\sin(\pi a)}
[\K_a(r)f_1(r')-\K_a(x)f_1(x')],
\end{align}
where $f_1$ is defined in Proposition \ref{pro2}(1), and hence the function $r\mapsto \K_a(r)f_1(r')$ is strictly increasing on $(0,1)$. It follows from (\ref{eq7}) that $f'_4<0$ if $0<p<1$ and $f'_4>0$ if $p>1$. Moreover, $f_4(1^-)=0$, and  the limiting value $f_4(0^+)=(1-p)R(a)/2$ follows from \cite[Theorem 5.5(3)]{AQVV}.

For the function $f_5$, we also let $x=r^p$. Then differentiating $f_{5}$ gives
\begin{align}\label{eq8}
\frac{d f_5(r)}{dr}=\frac{p\pi^2}{4r}
\left(\frac{1}{r'^2\K_a(r)^2}-\frac{1}{x'^2\K_a(x)^2}\right).
\end{align}
By \cite[Lemma 5.4(1)]{AQVV}, the function $r\mapsto r'^2\K_a(r)^2$ is strictly decreasing on $(0,1)$. It follows from (\ref{eq8}) that $f'_5<0$ if $0<p<1$ and $f'_5>0$ if $p>1$. Furthermore, $f_5(1^-)=0$, and  the limiting value $f_5(0^+)=(1-p)R(a)/2$ follows from  \cite[Theorem 5.5(2)]{AQVV}.
\end{proof}

\begin{pro}\label{pro4}
For all $a\in(0,1/2]$ and $r,t\in(0,1)$,
\begin{equation*}
m_a(r)+m_a(t)>m_a(rt),\quad \mu_a(r)+\mu_a(t)>\mu_a(rt).
\end{equation*}
\end{pro}
\begin{proof}
The first inequality was established in \cite[Theorem 2.5]{QV1999}. For the proof of the second one, we let $f_6(r)=\mu_a(r)+\mu_a(t)-\mu_a(rt)$ for arbitrarily given $a\in(0,1/2]$ and $t\in(0,1)$, and let $x=rt$, then
\begin{align*}
\frac{d f_6(r)}{dr}=\frac{\pi^2}{4r}
\left(\frac{1}{x'^2\K_a(x)^2}-\frac{1}{r'^2\K_a(r)^2}\right).
\end{align*}
It is clear to see that $f_6$ is strictly decreasing on $(0,1)$ and hence the second inequality follows from $f_6(1^-)=0$.
\end{proof}

\begin{pro}\label{pro3}
For each given $a\in(0,1/2]$ and $x,r\in(0,1)$ with $x<r$. Let $K\in(0,\infty)$, $s=\va_K^a(r)$ and $y=\va_K^a(x)$.
Define the functions $f_7$ and $f_8$ on $(0,\infty)$ by
\begin{align*}
f_7(K)=\frac{s'^2\K_a(s)^3}{y'^2\K_a(y)^3}, \quad
f_8(K)= \frac{\left(1+\frac{2a-1}{2(1-a)}s^2\right)\K_a(s)-\E_a(s)}
{\left(1+\frac{2a-1}{2(1-a)}y^2\right)\K_a(y)-\E_a(y)}.
\end{align*}
Then $f_7$ and $f_8$ are both strictly decreasing on $(0,\infty)$, with ranges $(0,1)$ and $(\mu_a(x)/\mu_a(r),\infty)$, respectively.
\end{pro}

\begin{proof} By logarithmic differentiation of $f_{7}$, we get
\begin{align*}
\frac{1}{f_7(K)}\frac{d f_7(K)}{dK}
&=-\frac{4}{\pi K\sin(\pi a)}s^2\K_a(s)\K_a(s')
 +\frac{12(1-a)}{\pi K\sin(\pi a)}\K_a(s') [\E_a(s)-s'^2\K_a(s)] \\
&\quad+\frac{4}{\pi K\sin(\pi a)}y^2\K_a(y)\K_a(y')
 -\frac{12(1-a)}{\pi K\sin(\pi a)}\K_a(y') [\E_a(y)-y'^2\K_a(y)] \\
&=\frac{4}{\pi K\sin(\pi a)} [f_9(y)-f_9(s)],
\end{align*}
where $f_9(r)=r^2\K_a(r')\K_a(r)-3(1-a)\K_a(r')\left[\E_a(r)-r'^2\K_a(r)\right]$.
Write $f_9(r)=r^2\K_a(r')f_{10}(r)$, where $f_{10}(r)=\K_a(r)-3(1-a)[\E_a(r)-r'^2\K_a(r)]/r^2$. By (\ref{Ka}) and (\ref{Ea}), we obtain
\begin{align*}
f_{10}(r)=\frac{\pi}{2} \s \left[ \left(1-\frac{3a(1-a)}{n+1}\right)
\frac{(a)_n(1-a)_n}{(n!)^2}\right] r^{2n},
\end{align*}
which shows that, for each given $a\in(0,1/2]$, $f_{10}$ is positive and strictly increasing on $(0,1)$. So is $f_9$ as it is a product of two positive and strictly increasing functions on $(0,1)$ by \cite[Lemma 5.4]{AQVV}. Therefore, $df_7(K)/dK<0$ for $K\in(0,\infty)$ since $y<s$,
and the monotonicity of $f_7$ on $(0,\infty)$ follows. Obviously, $f_7(0^+)=1$, and by (\ref{va+va'}) and \cite[Theorem 5.5(2)]{AQVV},
\begin{align*}
\lim_{K\to \infty}f_7(K)
&=\lim_{K\to \infty} \frac{\mu_a(y)^3s'^2}{\mu_a(s)^3y'^2}
=\frac{\mu_a(x)^3}{\mu_a(r)^3}
\lim_{K\to \infty} \left( \frac{\e^{\log s'+\mu_a(s')}}{\e^{\log y'+\mu_a(y')}} \e^{\mu_a(y')-\mu_a(s')} \right)^2 \\
&=\frac{\mu_a(x)^3}{\mu_a(r)^3} \lim_{K\to \infty}
\left( \e^{\mu_a(x')-\mu_a(r')} \right)^{2K}=0.
\end{align*}

Logarithmic differentiation of $f_{8}$ gives
\begin{align}\label{eq9}
\frac{1}{f_8(K)}\frac{d f_8(K)}{dK}
&=\frac{2s^2\K_a(s)\K_a(s')}{\pi K \sin(\pi a)}
\frac{\frac{a(2a-1)}{1-a}s'^2\K_a(s)+\E_a(s)}
{\left(1+\frac{2a-1}{2(1-a)}s^2\right)\K_a(s)-\E_a(s)} \nonumber\\
&\quad- \frac{2y^2\K_a(y)\K_a(y')}{\pi K \sin(\pi a)}
\frac{\frac{a(2a-1)}{1-a}y'^2\K_a(y)+\E_a(y)}
{\left(1+\frac{2a-1}{2(1-a)}y^2\right)\K_a(y)-\E_a(y)} \nonumber\\
&=\frac{2[f_{11}(s)-f_{11}(y)]}{\pi K \sin(\pi a)},
\end{align}
where $f_{11}(r)=\K_a(r')f_2(r)f_3(r)$, $f_2$ and $f_3$ are defined in Proposition \ref{pro2}, so that $f_{11}$ is strictly decreasing on $(0,1)$ as it is a product of three positive and strictly decreasing functions. It follows from $y<s$ and (\ref{eq9}) that $df_8(K)/dK<0$ for $K\in(0,\infty)$, hence the monotonicity of $f_8$ on $(0,\infty)$ follows. By Proposition \ref{pro2} and \cite[Theorem 5.5(2)]{AQVV},
\begin{align*}
\lim_{K\to 0^+}f_8(K)&=\lim_{K\to 0^+}\frac{s^2}{y^2}
=\lim_{K\to 0^+} \left( \frac{\e^{\log s+\mu_a(s)}}{\e^{\log y+\mu_a(y)}} \e^{\mu_a(y)-\mu_a(s)} \right)^2 \\
&=\lim_{K\to 0^+} \left(\e^{\mu_a(x)-\mu_a(r)} \right)^{\frac2K}=\infty,
\end{align*}
and by (\ref{Ramanujan1}),
\begin{align*}
\lim_{K\to \infty}f_8(K)=\lim_{K\to \infty} \frac{\left(1+\frac{2a-1}{2(1-a)}s^2\right)\frac{K\K_a(s')\K_a(r)}{\K_a(r')}-\E_a(s)}
{\left(1+\frac{2a-1}{2(1-a)}y^2\right)\frac{K\K_a(y')\K_a(x)}{\K_a(x')}-\E_a(y)}
=\frac{\mu_a(x)}{\mu_a(r)}.
\end{align*}
This completes the proof.
\end{proof}

\begin{theorem} \label{th5}
	Let $\lambda(a,r,t)$ and $\tau(a,r,t)$ be real functions defined on $(0,1/2]\times(0,1)\times(0,1)$. For arbitrarily given $(a,r,t)\in(0,1/2]\times(0,1)\times(0,1)$, define the functions $g_1$ and $g_2$ on $(0,\infty)$ by
	\begin{align*}
		g_1(K)=\frac{\va_K^a(r)\va_K^a(t)
			\left({\rm e}^{\la(a,r,t)}\right)^{1/K}}{\va_K^a(rt)}
		\quad \text{and} \quad
		g_2(K)=\frac{\va_{1/K}^a(r)\va_{1/K}^a(t)
			\left({\rm e}^{\tau(a,r,t)}\right)^K}{\va_{1/K}^a(rt)},
	\end{align*}
	respectively. Then we have the following conclusions:
	
	$(1)$ $g_1$ is strictly increasing (decreasing) on $(0,1]$ if and only if
	\begin{align*}
		\la(a,r,t)\leq m_a(r)+m_a(t)-m_a(rt) \,
		(\lambda(a,r,t)\geq \mu_a(r)+\mu_a(t)-\mu_a(rt), \text{respectively}).
	\end{align*}
	
	Moreover, $g_1$ is strictly decreasing on $(1,\infty)$ if and only if
	$\lambda(a,r,t)\geq m_a(r)+m_a(t)-m_a(rt)$.
	
	$(2)$ $g_2$ is strictly increasing on $(0,1]$ if and only if $\tau(a,r,t)\geq m_a(r)+m_a(t)-m_a(rt)$. Moreover, $g_2$ is strictly increasing (decreasing) on $(1,\infty)$ if and only if
	\begin{align*}
		\tau(a,r,t)\geq \mu_a(r)+\mu_a(t)-\mu_a(rt)\,
		(\tau(a,r,t)\leq m_a(r)+m_a(t)-m_a(rt), \text{respectively}).
	\end{align*}
\end{theorem}
\begin{proof} For part (1), let $x=rt$, $s=\va_K^a(r)$, $u=\va_K^a(t)$ and $y=\va_K^a(x)$. By logarithmic differentiation,
	\begin{align*}
		\frac{1}{g_1(K)}\frac{\partial g_1(K)}{\partial K}
		&=\frac{1}{s}\frac{\partial s}{\partial K}
		+\frac{1}{u}\frac{\partial u}{\partial K}
		-\frac{1}{y}\frac{\partial y}{\partial K}-\frac{\la(a,r,t)}{K^2} \\
		&=\frac{4s'^2\K_a(s)^2}{\pi^2 K^2}\mu_a(r)
		+\frac{4u'^2\K_a(u)^2}{\pi^2 K^2}\mu_a(t)
		-\frac{4y'^2\K_a(y)^2}{\pi^2 K^2}\mu_a(x)-\frac{\la(a,r,t)}{K^2}\\
		&=\frac{g_3(K,a,r,t)-\la(a,r,t)}{K^2},
	\end{align*}
	where
	\begin{align*}
		g_3(K,a,r,t)&=\frac{4}{\pi^2} \left[
		s'^2\K_a(s)^2\mu_a(r)+u'^2\K_a(u)^2\mu_a(t)-y'^2\K_a(y)^2\mu_a(x)\right].
	\end{align*}
	
	Differentiating $g_{3}$ with respect to $K$ gives
\begin{align}\label{g3'}
	&\frac{\pi^2}{4}\frac{\partial g_3(K,a,r,t)}{\partial K}
	=\left\{\frac{4(1-a)\K_a(s)[\E_a(s)-s'^2\K_a(s)]}{s}-2s\K_a(s)^2\right\}
	\mu_a(r) \frac{\partial s}{\partial K} \nonumber \\
	&\quad+ \left\{\frac{4(1-a)\K_a(u)[\E_a(u)-u'^2\K_a(u)]}{u}-2u\K_a(u)^2\right\}
	\mu_a(t) \frac{\partial u}{\partial K} \nonumber \\
	&\quad- \left\{\frac{4(1-a)\K_a(y)[\E_a(y)-y'^2\K_a(y)]}{y}-2y\K_a(y)^2\right\}
	\mu_a(x) \frac{\partial y}{\partial K} \nonumber \\
	=&\frac{16(1-a)\mu_a(x)^2 y'^2\K_a(y)^3}{\pi^2 K^2}
	\left[\left(1+\frac{2a-1}{2(1-a)}y^2\right)\K_a(y)-\E_a(y)\right]\nonumber\\
	&\quad-\frac{16(1-a)\mu_a(r)^2 s'^2\K_a(s)^3}{\pi^2 K^2}
	\left[\left(1+\frac{2a-1}{2(1-a)}s^2\right)\K_a(s)-\E_a(s)\right]\nonumber\\
	&\quad-\frac{16(1-a)\mu_a(t)^2 u'^2\K_a(u)^3}{\pi^2 K^2}
	\left[\left(1+\frac{2a-1}{2(1-a)}u^2\right)\K_a(u)-\E_a(u)\right]\nonumber\\
    =&\frac{16(1-a)y'^2\K_a(y)^3}{\pi^2 K^2}
	\left[\left(1+\frac{2a-1}{2(1-a)}y^2\right)\K_a(y)-\E_a(y) \right]
	g_4(K,a,r,t),
\end{align}
where
\begin{align}\label{g4}
	g_4(K,a,r,t)&=\mu_a(x)^2-\mu_a(r)^2 \frac{s'^2\K_a(s)^3}{y'^2\K_a(y)^3}
	\frac{\left(1+\frac{2a-1}{2(1-a)}s^2\right)\K_a(s)-\E_a(s)}
	{\left(1+\frac{2a-1}{2(1-a)}y^2\right)\K_a(y)-\E_a(y)} \nonumber \\
	&\quad -\mu_a(t)^2 \frac{u'^2\K_a(u)^3}{y'^2\K_a(y)^3} \frac{\left(1+\frac{2a-1}{2(1-a)}u^2\right)\K_a(u)-\E_a(u)}
	{\left(1+\frac{2a-1}{2(1-a)}y^2\right)\K_a(y)-\E_a(y)}.
\end{align}
According to Proposition \ref{pro3}, it is clear to see that the function $K\mapsto g_4(K,a,r,t)$ is strictly increasing on $(0,\infty)$, and
\begin{equation}\label{eq1}
	\lim_{K\to 0^+}g_4(K,a,r,t)=-\infty, \quad	\lim_{K\to \infty}g_4(K,a,r,t)=\mu_a(x)^2.
\end{equation}
Furthermore, one has
\begin{equation}\label{eq2}
g_4(1,a,r,t)=\frac{\pi^2}{4[\sin(\pi a)]^2}
\frac{g_5(a,r,t)} {x'^2 \K_a(x)^3
\left[\left(1+\frac{2a-1}{2(1-a)}x^2\right)\K_a(x)-\E_a(x)\right]}<0,
\end{equation}
since
\begin{align*}
	g_5(a,r,t)=&{x'}^2 \K_a(x)\K_a(x')^2
	\left[\left(1+\frac{2a-1}{2(1-a)}x^2\right)\K_a(x)-\E_a(x)\right]\\
	&-{r'}^2 \K_a(r)\K_a(r')^2
	\left[\left(1+\frac{2a-1}{2(1-a)}r^2\right)\K_a(r)-\E_a(r)\right]\\
	&-{t'}^2 \K_a(t)\K_a(t')^2
	\left[\left(1+\frac{2a-1}{2(1-a)}t^2\right)\K_a(t)-\E_a(t)\right]\\
	<&{x'}^2 \K_a(x)\K_a(x')^2
	\left[\left(1+\frac{2a-1}{2(1-a)}x^2\right)\K_fa(x)-\E_a(x)\right]\\
	&-{r'}^2 \K_a(r)\K_a(r') \K_a(x')
	\left[\left(1+\frac{2a-1}{2(1-a)}x^2\right)\K_a(x)-\E_a(x)\right]\\
	&-{t'}^2 \K_a(t)\K_a(t) \K_a(x')
	\left[\left(1+\frac{2a-1}{2(1-a)}x^2\right)\K_a(x)-\E_a(x)\right]\\
	=&\frac{\pi \sin(\pi a)}{2} \K_a(x')[m_a(x)-m_a(r)-m_a(t)] \\
	&\times \left[\left(1+\frac{2a-1}{2(1-a)}x^2\right)\K_a(x)-\E_a(x)\right]<0
\end{align*}
for each $a\in(0,1/2]$ and $r,t\in(0,1)$ by \cite[Lemma 5.4(1)]{AQVV}, Proposition \ref{pro2} (3) and Proposition \ref{pro4}.

It follows from \eqref{eq1} and \eqref{eq2} together with the monotonicity property of $K\mapsto g_4(K,a,r,t)$ that there exists $K_0\in(1,\infty)$ such that $g_4(K,a,r,t)<0$ for $K\in(0,K_0)$ and $g_4(K,a,r,t)>0$ for $K\in(K_0,\infty)$. This, in conjunction with (\ref{g3'}) and (\ref{g4}), shows that $K\mapsto g_3(K,a,r,t)$ is strictly decreasing on $(0,K_0)$ and then strictly increasing on $(K_0,\infty)$. Moreover, by Proposition \ref{pro4} and \cite[Lemma 5.4(1)]{AQVV},
\begin{equation*}
	g_{3}(0^+,a,r,t)=\mu_a(r)+\mu_a(t)-\mu_a(rt)>0, \quad
	g_3(1,a,r,t)=m_a(r)+m_a(t)-m_a(rt)>0,
\end{equation*}
\begin{equation*}
g_3(\infty,a,r,t)=0.	
\end{equation*}
Therefore, from the above the limiting values and the piecewise monotonicity of $K\mapsto g_3(K,a,r,t)$, we obtain
\begin{align*}
&\frac{\partial g_1}{\partial K}\geq0 \,\,  \text{for} \,\,  K\in(0,1]
	\Longleftrightarrow \la(a,r,t)\leq \inf_{K\in(0,1]} g_3(K,a,r,t)
	=m_a(r)+m_a(t)-m_a(rt), \\
&\frac{\partial g_1}{\partial K}\leq0 \,\,  \text{for} \,\,  K\in(0,1]
	\Longleftrightarrow \la(a,r,t)\geq \sup_{K\in(0,1]} g_3(K,a,r,t)
	=\mu_a(r)+\mu_a(t)-\mu_a(rt), \\
&\frac{\partial g_1}{\partial K}\leq0 \,\,  \text{for} \,\,  K\in[1,\infty)
	\Longleftrightarrow \la(a,r,t)\geq \sup_{K\in[1,\infty)} g_3(K,a,r,t)
	=m_a(r)+m_a(t)-m_a(rt),
\end{align*}
which yield the assertion of part (1). Part (2) follows from part (1) and the identity $g_2(K)=g_1(1/K)$.
\end{proof}

\begin{proof}[{\bf  Proof of Theorem \ref{th3}}] Taking $\lambda(a,r,t)=m_{a}(r)+m_{a}(t)-m_{a}(rt)$ in Theorem \ref{th5}, we obtain that $g_{1}(K)$ is strictly increasing on $(0,1]$, so that $g_{1}(K)\leq g_{1}(1)=\e^{\lambda(a,r,t)}=\e^{m_{a}(r)+m_{a}(t)-m_{a}(rt)}$. The asserted ``if" result of part (1) follows. Similarly, putting $\tau(a,r,t)=m_{a}(r)+m_{a}(t)-m_{a}(rt)$ and $\tau(a,r,t)=\mu_{a}(r)+\mu_{a}(t)-\mu_{a}(rt)$ into $g_{2}(K)$ in Theorem \ref{th5}, and then applying Theorem \ref{th5}(2), we obtain the ``if" result of part (2). In what follows we need to prove that ``only if" for part (1), since the proof of part (2) is similar. Denote $s=\va_K^a(r)$, $u=\va_K^a(t)$ and $y=\va_K^a(x)$. In (\ref{In7}), taking logarithm, rasing to power $K/(K-1)$ and letting $K\to1$, we obtain
\begin{align*}
	\lim_{K\to1}\frac{\log s+\log u-\log y}{1-1/K}\leq \alpha(a,r,t).
\end{align*}
Then by l'H\^{o}pital rule, the left-hand side of the above inequality is equal to
\begin{align*}
\lim_{K\to1}\frac{\log s+\log u-\log y}{1-1/K}=m_a(r)+m_a(t)-m_a(rt).
\end{align*}
This completes the proof.
\end{proof}

\begin{theorem}\label{th6}
Let $\xi(a,r)$ and $\rho(a,r)$ be  real functions defined on $(0,1/2]\times(0,1)$. For arbitrarily given $(a,r)\in(0,1/2]\times(0,1)$ and for $p\in(0,\infty)$, define the functions $g_6$ and $g_7$ on $[0,\infty)$ by
\begin{align*}
g_6(K)=\frac{\va_K^a(r)^p \left({\rm e}^{\xi(a,r)}\right)^{\frac1K}}{\va_K^a(r^p)}
\quad \text{and} \quad
g_7(K)=\frac{\va_{1/K}^a(r)^p\left({\rm e}^{\rho(a,r)}\right)^K }{\va_{1/K}^a(r^p)},
\end{align*}
respectively. Then we have the following conclusions:

$(1)$ If $0<p<1$, $g_6$ is strictly increasing (decreasing) on $(0,1]$ if and only if
\begin{align*}
\xi(a,r)\leq p\mu_a(r)-\mu_a(r^p) \,
(\xi(a,r)\geq pm_a(r)-m_a(r^p), \text{respectively}).
\end{align*}
Moreover, $g_6$ is strictly increasing on $[1,\infty)$ if and only if $\xi(a,r)\leq pm_a(r)-m_a(r^p)$.

If $p>1$, $g_6$ is strictly increasing (decreasing) on $(0,1]$ if and only if
\begin{align*}
\xi(a,r)\leq pm_a(r)-m_a(r^p)\,
(\xi(a,r)\geq p\mu_a(r)-\mu_a(r^p), \text{respectively}).
\end{align*}
Moreover, $g_6$ is strictly decreasing on $[1,\infty)$ if and only if $\xi(a,r)\geq pm_a(r)-m_a(r^p)$.

$(2)$ If $0<p<1$, $g_7$ is strictly decreasing on $(0,1]$ if and only if
$\rho(a,r)\leq pm_a(r)-m_a(r^p)$. Moreover, $g_7$ is strictly increasing (decreasing) on $[1,\infty)$ if and only if
\begin{align*}
\rho(a,r)\geq pm_a(r)-m_a(r^p)\,
 ( \rho(a,r)\leq p\mu_a(r)-\mu_a(r^p), \text{respectively}).
\end{align*}

If $p>1$,  $g_7$ is strictly increasing on $(0,1]$ if and only if
$\rho(a,r)\geq pm_a(r)-m_a(r^p)$. Moreover, $g_7$ is strictly increasing (decreasing) on $[1,\infty)$ if and only if
\begin{align*}
\rho(a,r)\geq p\mu_a(r)-\mu_a(r^p)\,
(\rho(a,r)\leq pm_a(r)-m_a(r^p), \text{respectively}).
\end{align*}
\end{theorem}

\begin{proof}
Since $g_7(K)=g_6(1/K)$, it suffices to prove the assertion of part (1). Let $x=r^p$, $s=\va_K^a(r)$ and $y=\va_K^a(x)$. Then logarithmic differentiation of $g_{6}$ with respect to $K$ yields
\begin{align}
\frac{1}{g_6(K)}\frac{\partial g_6(K)}{\partial K}
&=\frac{p}{s}\frac{\partial s}{\partial K}
  -\frac{1}{y}\frac{\partial y}{\partial K}-\frac{\xi(a,r)}{K^2}
=\frac{4p\mu_a(r)}{\pi^2 K^2} s'^2\K_a(s)^2
  -\frac{4\mu_a(x)}{\pi^2 K^2} y'^2\K_a(y)^2-\frac{\xi(a,r)}{K^2} \nonumber\\
&=\frac{g_8(K)-\xi(a,r)}{K^2}, \label{In8}
\end{align}
where $g_8(K)=4
\left[p\mu_a(r)s'^2\K_a(s)^2-\mu_a(x)y'^2\K_a(y)^2\right]/{\pi}^2$. By \cite[Lemma 5.4(1)]{AQVV}, one has
\begin{align}\label{In9}
\lim_{K\to 0^+}g_8(K)=p\mu_a(r)-\mu_a(r^p), \quad
\lim_{K\to 1}g_8(K)=pm_a(r)-m_a(r^p), \quad
\lim_{K\to \infty}g_8(K)=0.
\end{align}

Differentiating $g_{8}$ gives
\begin{align}\label{eq5}
\frac{\pi^2}{4}\frac{\partial g_8(K)}{\partial K}
&=p\mu_a(r) \frac{\partial s}{\partial K}
 \left\{\frac{4(1-a)\K_a(s)[\E_a(s)-s'^2\K_a(s)]}{s}-2s\K_a(s)^2\right\}
 \nonumber \\
&\quad- \mu_a(x) \frac{\partial y}{\partial K}
 \left\{\frac{4(1-a)\K_a(y)[\E_a(y)-y'^2\K_a(y)]}{y}-2y\K_a(y)^2\right\}
 \nonumber \\
&=\frac{16(1-a)y'^2 \K_a(y)^3 }{\pi^2 K^2}
   \left[\left(1+\frac{2a-1}{2(1-a)}y^2\right)\K_a(y)-\E_a(y)\right]g_9(K),
\end{align}
where
\begin{align}\label{eq3}
g_9(K)=\mu_a(x)^2-p\mu_a(r)^2
  \frac{s'^2\K_a(s)^3}{y'^2\K_a(y)^3} \frac{\left(1+\frac{2a-1}{2(1-a)}s^2\right)\K_a(s)-\E_a(s)}
  {\left(1+\frac{2a-1}{2(1-a)}y^2\right)\K_a(y)-\E_a(y)}.
\end{align}

Next, we divide the proof into two cases.

Case I: $p>1$. Then $x<r$ and hence $y<s$, it follows from (\ref{eq3}) and Proposition \ref{pro3} that $g_9$ is strictly increasing on $(0,\infty)$, and
\begin{align}\label{eq4}
\lim_{K\to 0^+}g_9(K)=-\infty, \quad
\lim_{K\to \infty}g_9(K)=\mu_a(x)^2.
\end{align}
Furthermore, one can also obtain
\begin{align}\label{g9(1)}
g_9(1)=\mu_a(x)^2-\frac{\pi^2p}{4[\sin(\pi a)]^2}
  \frac{r'^2\K_a(r)\K_a'(r)^2}{x'^2\K_a(x)^3} \frac{\left(1+\frac{2a-1}{2(1-a)}r^2\right)\K_a(r)-\E_a(r)}
  {\left(1+\frac{2a-1}{2(1-a)}x^2\right)\K_a(x)-\E_a(x)}<0
\end{align}
for all $a\in(0,1/2]$, $r\in(0,1)$ and $p\in(1,\infty)$. As a matter of fact, by \cite[Lemma 5.4(1)]{AQVV}, Proposition \ref{pro2}(3) and Proposition \ref{pro1},
\begin{align*}
 &\frac{4[\sin(\pi a)]^2}{\pi^2}x'^2\K_a(x)^3
 \left[\left(1+\frac{2a-1}{2(1-a)}x^2\right)\K_a(x)-\E_a(x)\right] g_9(1)
 \nonumber \\
=&x'^2\K_a(x)\K_a'(x)^2
\left[\left(1+\frac{2a-1}{2(1-a)}x^2\right)\K_a(x)-\E_a(x)\right]\nonumber\\
 &-p r'^2\K_a(r)\K_a'(r)^2
\left[\left(1+\frac{2a-1}{2(1-a)}r^2\right)\K_a(r)-\E_a(r)\right]\nonumber\\
<&x'^2\K_a(x)\K_a'(x)^2
\left[\left(1+\frac{2a-1}{2(1-a)}x^2\right)\K_a(x)-\E_a(x)\right]\nonumber\\
 &-p r'^2\K_a(r)\K_a'(r)\K_a'(x)
\left[\left(1+\frac{2a-1}{2(1-a)}x^2\right)\K_a(x)-\E_a(x)\right]\nonumber\\
=&\frac{\pi \sin(\pi a)}{2}\K_a'(x) [m_a(r^p)-pm_a(r)]
  \left[\left(1+\frac{2a-1}{2(1-a)}x^2\right)\K_a(x)-\E_a(x)\right]<0
\end{align*}
for all $a\in(0,1/2]$, $r\in(0,1)$ and $p\in(1,\infty)$.

Equation \eqref{eq4}, inequality \eqref{g9(1)} and the monotonicty property of $g_{9}$ on $(0,\infty)$ show that there exists $K_1\in(1,\infty)$ such that $g_9(K)<0$ for $K\in(0,K_1)$, and $g_9(K)>0$ for $K\in(K_1,\infty)$. This, together with (\ref{eq5}), leads to the conclusion that $g_8$ is strictly decreasing on $(0,K_1)$, and then strictly increasing on $(K_1,\infty)$. Combining with \eqref{In8}, \eqref{In9} and Proposition \ref{pro1}, we have
\begin{align*}
&\frac{\partial g_6}{\partial K}\geq0 \,\,  \text{for} \,\,  K\in(0,1]
 \Longleftrightarrow \xi(a,r)\leq \inf_{K\in(0,1]} g_8(K)
 =pm_a(r)-m_a(r^p), \\
&\frac{\partial g_6}{\partial K}\leq0 \,\,  \text{for} \,\,  K\in(0,1]
 \Longleftrightarrow \xi(a,r)\geq \sup_{K\in(0,1]} g_8(K)
 =p\mu_a(r)-\mu_a(r^p), \\
&\frac{\partial g_6}{\partial K}\leq0 \,\,  \text{for} \,\,  K\in[1,\infty)
\Longleftrightarrow \xi(a,r)\geq \sup_{K\in[1,\infty)} g_8(K)
 =pm_a(r)-m_a(r^p),
\end{align*}
which yield the assertion of part (1) for $p>1$.

Case II:  $0<p<1$. Then $x>r$ and hence $y>s$, it follows from (\ref{eq3}) and Proposition \ref{pro3} that $g_9$ is strictly decreasing on $(0,\infty)$, and
\begin{align}\label{eq6}
\lim_{K\to 0^+}g_9(K)=\mu_a(x)^2, \quad
\lim_{K\to \infty}g_9(K)=-\infty.
\end{align}
In this case,
\begin{align}\label{In6}
g_9(1)=\mu_a(x)^2-\frac{\pi^2p}{4[\sin(\pi a)]^2}
  \frac{r'^2\K_a(r)\K_a'(r)^2}{x'^2\K_a(x)^3} \frac{\left(1+\frac{2a-1}{2(1-a)}r^2\right)\K_a(r)-\E_a(r)}
  {\left(1+\frac{2a-1}{2(1-a)}x^2\right)\K_a(x)-\E_a(x)}>0
\end{align}
for all $a\in(0,1/2]$, $r\in(0,1)$ and $p\in(1,\infty)$, since
\begin{align*}
 &\frac{4[\sin(\pi a)]^2}{\pi^2}x'^2\K_a(x)^3
 \left[\left(1+\frac{2a-1}{2(1-a)}x^2\right)\K_a(x)-\E_a(x)\right] g_9(1)\\
>&x'^2\K_a(x)\K_a'(x)\K_a'(r)
  \left[\left(1+\frac{2a-1}{2(1-a)}r^2\right)\K_a(r)-\E_a(r)\right] \\
 &-p r'^2\K_a(r)\K_a'(r)^2
  \left[\left(1+\frac{2a-1}{2(1-a)}r^2\right)\K_a(r)-\E_a(r)\right] \\
=&\frac{\pi \sin(\pi a)}{2}\K_a'(r) [m_a(r^p)-pm_a(r)]
  \left[\left(1+\frac{2a-1}{2(1-a)}r^2\right)\K_a(r)-\E_a(r)\right]>0
\end{align*}
for all $a\in(0,1/2]$, $r\in(0,1)$ and $p\in(1,\infty)$.

Equation \eqref{eq6}, inequality \eqref{In6} and the monotonicity property of $g_9$ imply that there exists $K_2\in(1,\infty)$ such that $g_9(K)>0$ for $K\in(0,K_2)$, and $g_9(K)<0$ for $K\in(K_2,\infty)$. This, in conjunction with \eqref{eq5}, implies that $g_8$ is strictly increasing on $(0,K_2)$, and then strictly decreasing on $(K_2,\infty)$. Therefore, from \eqref{In8}, \eqref{In9} and Proposition \ref{pro1}, one has
\begin{align*}
&\frac{\partial g_6}{\partial K}\geq0 \,\,  \text{for} \,\,  K\in(0,1]
 \Longleftrightarrow \xi(a,r)\leq \inf_{K\in(0,1]} g_8(K)
 =p\mu_a(r)-\mu_a(r^p), \\
&\frac{\partial g_6}{\partial K}\leq0 \,\,  \text{for} \,\,  K\in(0,1]
 \Longleftrightarrow \xi(a,r)\geq \sup_{K\in(0,1]} g_8(K)
 =pm_a(r)-m_a(r^p), \\
&\frac{\partial g_6}{\partial K}\geq0 \,\,  \text{for} \,\,  K\in[1,\infty)
 \Longleftrightarrow \xi(a,r)\leq \inf_{K\in[1,\infty)} g_8(K)
 =pm_a(r)-m_a(r^p),
\end{align*}
which yield the assertion of part (1) for $0<p<1$.

This completes the proof.
\end{proof}

\begin{proof}[{\bf Proof of Theorem \ref{th4}}] It suffices to prove part (1), because the proof of part (2) is similar. For $0<p<1$,   since the function $g_{6}(K)$  with $\xi(a,r)=pm_{a}(r)-m_{a}(r^p)$ in Theorem \ref{th6}(1) is strictly increasing on $[1,\infty)$, then $g_{6}(K)\geq g_{6}(1)=\e^{pm_{a}(r)-m_{a}(r^p)}$, from which the first inequality in \eqref{th4eq1} holds. On the other hand, in the particular case $\rho(a,r)=pm_{a}(r)-m_{a}(r^p)$ ($\rho(a,r)=p\mu_{a}(r)-\mu_{a}(r^p)$, respectively), the function $g_{7}(K)$ in Theorem \ref{th6}(2) is strictly increasing (decreasing, respectively) on $[1,\infty)$, so that $g_{7}(K)\geq g_{7}(1)=\e^{pm_{a}(r)-m_{a}(r^p)}$ ($g_{7}(K)\leq g_{7}(1)=\e^{p\mu_{a}(r)-\mu_{a}(r^p)}$, respectively) for $K\in[1,\infty)$, from which the left-hand (right-hand, respectively) of the double inequality in \eqref{th4eq1} holds. Until now, the asserted ``if" results of part (1) have been proved. To obtain the ``only if", taking logarithmic of two inequalities in \eqref{th4eq1}, and then multiplying by $K/(K-1)$ on the two sides of the first inequality and $1/(1-K)$ for the second one, and finally taking the limiting values, we get
\begin{equation*}
	\delta(a,r)\leq \lim_{K\rightarrow 1^+}\frac{p\log{\va_K^a(r)}-\log{\va_K^a(r^p)}}{1-1/K}=pm_{a}(r)-m_{a}(r^p),
\end{equation*}
\begin{equation*}
		\zeta(a,r)\geq \lim_{K\rightarrow 1^+}\frac{p\log{\va_{1/K}^a(r)}-\log{\va_{1/K}^a(r^p)}}{1-K}=pm_{a}(r)-m_{a}(r^p),
\end{equation*}
\begin{equation*}		
	\eta(a,r)\leq \lim_{K\rightarrow \infty}\frac{p\log{\va_{1/K}^a(r)}-\log{\va_{1/K}^a(r^p)}}{1-K}=p\mu_{a}(r)-\mu_{a}(r^p).
\end{equation*}
\end{proof}

\begin{remark}\rm{} If $a=1/2$, then Theorems \ref{th3} and \ref{th4} reduce to Theorems \ref{th1} and \ref{th2}, respectively. Letting $t\to 0$, then Theorem \ref{th5} reduces to  \cite[Theorem 3]{WZC2007}.
\end{remark}

\begin{remark}\rm{} It is not difficult from Theorems \ref{th3} and \ref{th4} to see that one can obtain new inequalities for $\varphi_{K}^{a}(r)$ by bounding the functions $m_{a}(r)$ and $\mu_{a}(r)$. For example, it was proved in \cite[Theorem 3.5]{WCZ2019} and \cite[Theorem 1]{WZC2007} that
	\begin{equation}\label{reinequ}
	\frac{R(a)}{2}r'(1-r)<m_{a}(r)+\log{r}<\frac{R(a)}{2}{r'}^{2-\lambda}<	\frac{R(a)}{2}r'
	\end{equation}
for all $a\in(0,1/2]$ and $r\in(0,1)$ with $\lambda=2a(1-a)(1-a+a^2)/[a^2+(1-a)^2]$. Then use of \eqref{reinequ} and the decreasing property of $r\mapsto m_{a}(r)+\log{r}$ on $(0,1)$ leads to the conclusion that, for all $(a,r,t)\in(0,1/2]\times(0,1)\times(0,1)$,
\begin{align*}
&m_{a}(r)+m_{a}(t)-m_{a}(rt)=m_{a}(r)+\log{r}+m_{a}(t)+\log{t}-[m_{a}(rt)+\log{rt}]\\
\leq& m_{a}(r)+\log{r}+m_{a}(t)+\log{t}-\frac{1}{2}\left[m_{a}(r)+\log{r}+m_{a}(t)+\log{t}\right]\\
=&\frac{1}{2}\left[m_{a}(r)+\log{r}+m_{a}(t)+\log{t}\right]\leq \frac{R(a)}{4}\left[{r'}^{2-\lambda}+{t'}^{2-\lambda}\right]
\end{align*}
and, for $p>1$,
\begin{align*}
&pm_{a}(r)-m_{a}(r^p)=p\left[m_{a}(r)+\log{r}\right]-\left[m_{a}(r^p)+\log{r^p}\right]\\
\leq &(p-1)\left[m_{a}(r)+\log{r}\right]\leq \frac{R(a)(p-1)}{2}{r'}^{2-\lambda}
\end{align*}
and while for $0<p<1$,
\begin{align*}
	&pm_{a}(r)-m_{a}(r^p)=p\left[m_{a}(r)+\log{r}\right]-\left[m_{a}(r^p)+\log{r^p}\right]\\
	\geq &(p-1)\left[m_{a}(r)+\log{r}\right]\geq \frac{R(a)(p-1)}{2}{r'}^{2-\lambda},
\end{align*}
so that inequalities
\begin{equation*}
	\frac{\va_K^a(r)\va_K^a(t)}{\va_K^a(rt)}
	<{\e}^{\frac{R(a)}{4}({r'}^{2-\lambda}+{t'}^{2-\lambda})\left(1-\frac1K\right)},
\end{equation*}
and
\begin{equation*}
	\frac{\va_K^a(r)^p}{\va_K^a(r^p)}>{\e}^{\frac{R(a)}{2}(p-1){r'}^{2-\lambda}\left(1-\frac1K\right)}, \quad p\in(0,1)
\end{equation*}
\begin{equation*}
		\frac{\va_K^a(r)^p}{\va_K^a(r^p)}<{\e}^{\frac{R(a)}{2}(p-1){r'}^{2-\lambda}\left(1-\frac1K\right)}, \quad p\in(1,\infty)
\end{equation*}
hold for all $a\in(0,1/2]$, $r\in(0,1)$, $t\in(0,1)$ and $K\in[1,\infty)$.
\end{remark}

\end{document}